\documentclass[12pt,twoside,reqno]{amsart}
\usepackage{amsfonts}
\usepackage{amssymb}
\usepackage[notcite,notref]{%showkeys
}

\numberwithin{equation}{section}

\newcommand{\esssup}{\mathop{\mathrm{ess\,sup}}}

\newtheorem{teo}{Theorem}[section]
\newtheorem{prop}[teo]{Proposition}

\theoremstyle{definition}

\newtheorem{rem}[teo]{Remark}

\numberwithin{equation}{section}

\def\a{\alpha}
\def\b{\beta}
\def\l{\lambda }

\def\o{\omega}
\def\R{\mathbb{R}}

\def\d{\delta}
\def\e{\varepsilon}

\def\f{\varphi}
\def\s{\sigma}
\def\mes{\operatorname{mes}}

\def\esup{\operatorname{ess ~sup}}
\def\Lip{\operatorname{Lip}}
\begin{document}

%%%%%%%%%%%%%%%%%%%%%%%%%%%%%%%%%%%%%%%%%%%%%%%%%%%%%%%%%%%%%%%%%%%%%%

\title[Sections of functions]{Sections of functions and Sobolev type inequalities}

\author[V.I. Kolyada]{V.I. Kolyada}
\address{Department of Mathematics\\
Karlstad University\\
Universitetsgatan 2 \\
651 88 Karlstad\\
SWEDEN} \email{viktor.kolyada@kau.se}



\begin{abstract} We study functions of two variables whose sections by the lines parallel to the coordinate axis satisfy
Lipschitz condition of the order $0<\a\le 1.$ We prove that if for a function $f$ the $\operatorname{Lip} \a-$ norms of these sections belong to the Lorentz space $L^{p,1}(\R) \,(p=1/\a),$ then $f$ can be modified on a set of measure zero so as
to become bounded and uniformly continuous on $\R^2.$ For $\a=1$ this gives
an extension of  Sobolev's theorem on continuity of functions of
the space $W_1^{2,2}(\R^2)$. We show that the exterior
$L^{p,1}-$ norm cannot be replaced by a weaker Lorentz norm $L^{p,q}$ with $q>1$.

\end{abstract}

\dedicatory{Dedicated to O.V. Besov  on the occasion of his 80th
birthday}

%%%%%%%%%%%%%%%%%%%%%%%%%%%%%%%%%%%%%%%%%%%%%%%%%%%%%%%%%%%%%%%%%%%%%%

\maketitle

\date{}

\maketitle

\vskip 8pt

\section{Introduction}

The classical embedding with limiting exponent
$$
W_1^1(\R^n)\subset L^{n/(n-1)}(\R^n)
$$
for the Sobolev space $W_1^1$ was proved independently by
Gagliardo \cite{Ga} and Nirenberg  \cite{Nir}. Gagliardo's
 approach was based on estimates of
certain mixed norms. A refinement of these estimates and a further
development of Gagliardo's method were obtained by  Fournier
\cite{Four}. Different extensions of these results and their
applications have been  studied, e.g., in  the works \cite{Alg},
\cite{AK}, \cite{BF},  \cite{K2006}, \cite{K2012}, \cite{Milman}.

In what follows we consider functions of two variables. Let a
function $f$ be defined on $\R^2.$ For any fixed $x\in \R,$  the
$x$-section of $f$ (denoted by $f_x$) is the function of the
variable $y$ defined by $f_x(y)=f(x,y)\,\, (y\in \R).$ Similarly,
for a fixed $y\in \R,$ the $y-$section of $f$ is the function
$f_y(x)=f(x,y)\,\, (x\in \R)$ of the variable $x$. The
Gagliardo-Fournier mixed norm space is defined as
\begin{equation}\label{gag_four}
L^1_x[L^\infty_y]\cap L^1_y[L^\infty_x]\equiv
L^1[L^\infty]_{\operatorname{sym}}.
\end{equation}
Thus, for a function $f\in L^1[L^\infty]_{\operatorname{sym}}$
almost all linear sections are essentially bounded, and the
$L^\infty-$ norms of these sections belong to $L^1(\R).$ As it was
shown in \cite{Ga}, \cite{Four}, these conditions imply certain
integrability properties of $f.$

The question studied in this paper is a part of a general problem
which can be formulated as follows: How do the smoothness
conditions imposed on linear sections of a function affect its
{\it global} continuity properties ? More precisely,  we shall
consider mixed norm spaces of functions whose linear sections
satisfy Lipschitz conditions.

For any function $\f$ on $\R,$ set $\Delta_h\f(t)=\f(t+h)-\f(t).$
Let $\a\in (0,1].$ Denote by $\operatorname{Lip} \a$ the class of
all functions $\f\in L^\infty(\R)$ such that
$$
||\f||^*_{\operatorname{Lip} \a}=\sup_{h>0}
h^{-\a}||\Delta_h\f||_\infty<\infty.
$$
We set also
$$
||\f||_{\operatorname{Lip}
\a}=||\f||_\infty+||\f||^*_{\operatorname{Lip} \a}.
$$

Some mixed norm norm spaces of functions with smoothness
conditions on sections were studied in the dissertation
\cite{Alg}. In particular, it was proved in \cite[Theorem
8.13]{Alg} that every function
$$
f\in L^p_x[(\operatorname{Lip}\a)_y]\cap
L^p_y[(\operatorname{Lip}\a)_x],\quad\mbox{where}\quad 0<\a\le 1,
\,1/\a<p<\infty,
$$
is equivalent to a bounded and uniformly continuous function on
$\R^2.$ However, the limiting case $p=1/\a$ was left open.

 The main objective of the present paper is to study this limiting case. Our interest to this problem
     is partly motivated by its close relation to embedding of the Sobolev space $W_1^{2,2}(\R^2)$.
Note that this relation is similar to the one between  embeddings
of Gagliardo-Fournier space (\ref{gag_four}) and the Sobolev space
$W_1^1(\R^2).$

Denote by $W_1^{2,2}(\R^2)$ the space of all functions $f\in
L^1(\R^2)$ for which  pure distributional partial derivatives of the second order
$D_1^2f$ and $D_2^2f$ exist and belong to $L^1(\R^2)$.  It is
well known that this doesn't imply the existence of mixed
derivatives.

Sobolev's theorem asserts that every function $f\in
W_1^{2,2}(\R^2)$ can be modified on a set of measure zero so as to
become uniformly continuous and bounded on $\R^2$ (see
\cite[Theorems 10.1 and  10.4]{BIN}).

We have the following embedding
\begin{equation}\label{embed_1}
W_1^{2,2}(\R^2)\subset L^1_x[(\operatorname{Lip} 1)_y]\cap
L^1_y[(\operatorname{Lip} 1)_x]\equiv L^1[\operatorname{Lip}
1]_{\operatorname{sym}}
\end{equation}
(see Proposition \ref{EMBED} below).

Let $\a\in (0,1].$ Assume that almost all $x-$sections and almost
all $y-$sections of $f$ belong to $\operatorname{Lip} \a$. We
consider the functions
\begin{equation}\label{mathcal}
\mathcal{N}_{\a}^{(1)}f(y)=||f_y||_{\operatorname{Lip}
\a}\quad\mbox{and}\quad
\mathcal{N}_{\a}^{(2)}f(x)=||f_x||_{\operatorname{Lip} \a}.
\end{equation}
The integrability properties  of these functions provide important
characteristics of smoothness of sections. A natural  measure of
these properties can be obtained in terms of rearrangements of functions (\ref{mathcal}) and their Lorentz norms (cf.
\cite{AK}).

Denote by $S_0(\R^n)$ the class of all measurable and almost
everywhere finite functions $f$ on $\mathbb{R}^n$ such that
\begin{equation*}
\lambda_f (y) \equiv | \{x \in \mathbb{R}^n : |f(x)|>y \}| <
\infty\quad \text{for each $y>0$}.
\end{equation*}
{\it{A non-increasing rearrangement}} of a function $f \in
S_0(\mathbb{R}^n)$ is a non-negative and non-increasing function
$f^*$ on $\mathbb{R}_+ \equiv (0, + \infty)$ which is
equimeasurable with $|f|$, that is,  for any $y>0$
\begin{equation*}
|\{t\in \mathbb{R}_+ : f^*(t)>y\}|= \lambda_f (y)
\end{equation*}
(see \cite[Ch. 1]{BS}). The Lorentz space $L^{p,q}(\R^n)\,\,
 (p,q \in [1,\infty))$ is defined as the class  of all functions  $f\in
S_0(\R^n)$ such that
$$
||f||_{L^{p,q}}\equiv||f||_{p,q}=\left( \int_0^\infty \left(
t^{1/p} f^*(t) \right)^q \frac{dt}{t} \right)^{1/q} < \infty.
$$
We have that $||f||_{p,p}=||f||_p.$ For a fixed $p$, the Lorentz
spaces $L^{p,q}$ strictly increase as the secondary index $q$
increases; that is, the strict embedding $L^{p,q}\subset
L^{p,r}~~~(q<r)$ holds (see \cite[Ch. 4]{BS}).

For any $1\le p<\infty$, denote
$$
\mathcal U_p(\R^2)=L^{p,1}_x\left[\left(\operatorname{Lip}
\frac1p\right)_y\right]\bigcap
L^{p,1}_y\left[\left(\operatorname{Lip}
\frac1p\right)_x\right]\equiv L^{p,1}\left[\operatorname{Lip}
\frac1p\right]_{\operatorname{sym}}.
$$
For a function $f\in \mathcal U_p(\R^2),$ set
$$
||f||_{\mathcal
U_p(\R^2)}=||\mathcal{N}_{1/p}^{(1)}f||_{p,1}+||\mathcal{N}_{1/p}^{(2)}f||_{p,1}.
$$

By (\ref{embed_1}), $W_1^{2,2}(\R^2)\subset \mathcal U_1(\R^2).$

The  paper is organized as follows.  In Section 2 we prove  that
every function $f\in \mathcal U_p(\R^2)\,\,(1\le p<\infty)$ can be
modified on a set of measure zero so as to become bounded and
uniformly continuous on $\R^2$, and we give an estimate of the
modulus of continuity of the modified function. This is the main
result of the paper. In particular, it provides a generalization
of the Sobolev theorem on continuity of functions in
$W_1^{2,2}(\R^2).$ We show also that the result is optimal in the
sense that the exterior $L^{p,1}-$ norm cannot be replaced by a
weaker Lorentz norm $L^{p,q}$ with $q>1$. In Section 3 we show
that the spaces $\mathcal U_p(\R^2)$ increase as $p$ increases,
and we prove embedding (\ref{embed_1}).

\vskip 5pt

\section{Continuity}

We recall some definitions and results which will be used in the sequel.

If $f$ is a continuous function on $\R^2$, then its modulus of
continuity is defined by
$$
\o(f;\d)=\sup_{0\le h,k\le\d}|f(x+h,y+k)-f(x,y)|.
$$

For any function $f\in S_0(\R^n)$, denote
$$
f^{**}(t)= \frac{1}{t} \int_0^t f^*(u) du.
$$

We shall use  the following inequality. Let $g\in S_0(\R^n)$. Then
for any $0<s<t\le\infty$
\begin{equation}\label{rear}
g^*(s)-g^*(t)\le \frac{1}{\ln 2}\int_{s/2}^t
[g^*(u)-g^*(2u)]\frac{du}{u}.
\end{equation}
Indeed,
$$
\begin{aligned}
&\int_{s/2}^t[g^*(u)-g^*(2u)]\frac{du}{u}\\
&=\int_{s/2}^s g^*(u)\frac{du}{u}-\int_t^{2t}g^*(u)\frac{du}{u}
\ge \left[g^*(s)-g^*(t)\right]\ln 2.
\end{aligned}
$$

It is easy to see that for any $g\in S_0(\R^n)$
\begin{equation}\label{rear1}
\lim_{t\to +\infty} g^*(t)=0.
\end{equation}

Let $E\subset \R^2$ be a measurable set. For any $y\in
\R,$ denote by $E(y)$ the $y-$section of the set $E,$ that is
$$
E(y)=\{x\in \R: (x,y)\in E\}.
$$
\textit{The essential projection} of $E$ onto the $y-$axis  is
defined to be the set $\Pi$ of all  $y\in \R$ such that $E(y)$ is
measurable and $\mes_1 E(y)>0.$ Since the function $y\mapsto
\mes_1 E(y)$ is measurable, the essential projection is a
measurable set in $\R$. Similarly we define the $x-$ sections and
the essential projection of $E$ onto the $x-$ axis.

 \begin{teo}\label{Second}Let $1\le p<\infty.$ Then every function $f\in \mathcal U_p(\R^2)$
 belongs to $S_0(\R^2)$ and can be modified on a set of measure zero so as to become uniformly
 continuous and bounded on $\R^2.$
 Moreover, if
 $$
\f_1(y)=||f_y||^*_{\operatorname{Lip}\frac1p}\quad\mbox{and}\quad
\f_2(x)=||f_x||^*_{\operatorname{Lip}\frac1p},
$$
then
 \begin{equation}\label{bound}
 ||f||_\infty\le c \left(||\f_1||_{p,1}+||\f_2||_{p,1}\right)
 \end{equation}
 and for the modified function $\bar {f}$ we have that
\begin{equation}\label{omega}
\omega(\bar f;\delta)\le c \int_0^\delta
\left[\f_1^*(t)+\f_2^*(t)\right]t^{1/p-1}\,dt.
\end{equation}

\end{teo}
\begin{proof} Let $f\in \mathcal U_p(\R^2)$. First we show that $f\in S_0(\R^2).$ Set
$$
\psi_1(y)= ||f_y||_\infty, \quad \psi_2(x)= ||f_x||_{\infty}.
$$
Then
\begin{equation}\label{infty}
|f(x,y)|\le \min\left(\psi_1(y),\psi_2(x)\right)\quad\mbox{for
almost all} \quad (x,y)\in \R^2.
\end{equation}
It follows that for any $\a>0$ the set $\{(x,y):|f(x,y)|>\a\}$ is
contained in the cartesian product
$\{x:\psi_2(x)>\a\}\times\{y:\psi_1(y)>\a\}$, except a subset of
measure zero. Thus, $\l_f(\a)\le \l_{\psi_1}(\a)\l_{\psi_2}(\a).$
Since $\psi_1,\psi_2\in L^{p,1}(\R),$ this implies that
$\l_f(\a)<\infty$ for any $\a>0.$

Now we shall prove that for any $t>0$
\begin{equation}\label{3_1}
f^*(t)-f^*(2t)\le c
t^{1/(2p)}\left(\f_1^*\left(\frac{\sqrt{t}}{2}\right)+\f_2^*\left(\frac{\sqrt{t}}{2}\right)\right).
\end{equation}
Fix $t>0$. There exist a set $A$ of type $F_\s$ and a set $B$ of
type $G_\d$ such that $A\subset B,$ $\mes_2 A=t,$ $\mes_2 B=2t$,
and
$$
|f(x,y)|\ge f^*(t)\quad\mbox{for all}\quad (x,y)\in A,
$$
$$
|f(x,y)|\le f^*(2t)\quad\mbox{for all}\quad (x,y)\not\in B.
$$
At least one of the essential projections of the set $A$ onto the
coordinate axes has the one-dimensional measure not smaller than
$\sqrt{t}.$ Assume that the projection onto the $y-$axis has this
property, and denote this projection by $P.$ For any $y\in P$, set
$\b(y)=\mes_1 B(y).$ We have
$$
\int_P \b(y)\,dy \le \mes_2 B=2t.
$$
This implies that $\mes_1\{y\in P: \b(y)\ge 4\sqrt{t}\}\le
\sqrt{t}/2$.
 Hence, there exists a subset $Q\subset P$ of type $F_\s$ such that $\mes_1 Q\ge\sqrt{t}/2,$
\begin{equation}\label{b(y)}
\b(y)\le 4\sqrt{t}, \quad\mbox{and}\quad f_y\in \Lip \frac1p
\end{equation}
for any $y\in Q.$ Fix $y\in Q.$ Observe that the section $A(y)$
has a positive one-dimensional measure.  Further, there exists
$h\in  (0,8\sqrt{t}]$ such that
$$
\mes_1\{x\in A(y): x+h\not\in B(y)\}>0.
$$
Indeed, otherwise for any $h\in  (0,8\sqrt{t}]$ we would have that
$$
\chi_{B(y)}(x+h)=1 \quad\mbox{for almost all} \quad x\in A(y),
$$
where $\chi_{B(y)}$ is the characteristic function of the set
$B(y).$ Thus, for any $h\in  (0,8\sqrt{t}]$
$$
\int_{A(y)} \chi_{B(y)}(x+h)\,dx=\mes_1 A(y).
$$
Integrating this equality with respect to $h$ and interchanging
the order of integrations, we obtain
\begin{equation}\label{contr}
\int_{A(y)}\,dx
\int_0^{8\sqrt{t}}\chi_{B(y)}(x+h)\,dh=8\sqrt{t}\mes_1 A(y).
\end{equation}
But $y\in Q,$ and therefore, by the first  condition in
(\ref{b(y)})
$$
\int_0^{8\sqrt{t}}\chi_{B(y)}(x+h)\,dh\le \int_\R
\chi_{B(y)}(u)\,du=\b(y)\le 4\sqrt{t}.
$$
This implies that the left-hand side of (\ref{contr}) doesn't
exceed $4\sqrt{t}\mes_1 A(y),$ and we obtain a contradiction since
$\mes_1 A(y)>0.$ Thus,
 there exists $h\in (0,8\sqrt{t}]$ such that
$|f(x+h,y)|\le f^*(2t)$ for all $x$ from some subset $A'(y)\subset
A(y)$ with $\mes_1 A'(y)>0.$ In addition, $|f(x,y)|\ge f^*(t)$ for
all $x\in A(y).$ Thus, we have
$$
f^*(t)-f^*(2t)\le |f(x,y)-f(x+h,y)|\quad\mbox{for any}\quad x\in
A'(y).
$$
Since $\mes_1 A'(y)>0,$ this implies that
$$
f^*(t)-f^*(2t)\le \esssup\limits_{x\in R} |f(x,y)-f(x+h,y)|.
$$
Using also the second condition in (\ref{b(y)}),
we obtain that
$$
f^*(t)-f^*(2t)
\le h^{1/p}||f_y||^*_{\Lip \frac1p}\le (8\sqrt{t})^{1/p}\f_1(y)
$$
for any $y\in Q.$ Since $\mes_1 Q \ge\sqrt{t}/2,$ by the
definition of the non-increasing rearrangement we have that
$\inf_{y\in Q}\f_1(y)\le \f^*(\sqrt{t}/2).$ Hence,
$$
f^*(t)-f^*(2t)\le c
t^{1/(2p)}\f_1^*\left(\frac{\sqrt{t}}{2}\right).
$$
Similarly, in the case when the projection of $A$ onto the
$x-$axis has the one-dimensional measure at least $\sqrt{t},$ we
have the estimate
$$
f^*(t)-f^*(2t)\le c
t^{1/(2p)}\f_2^*\left(\frac{\sqrt{t}}{2}\right).
$$
Thus, we have proved inequality (\ref{3_1}). Using (\ref{3_1}),
(\ref{rear}), and (\ref{rear1}), we get
$$
\begin{aligned}
&||f||_\infty\le 2\int_0^\infty [f^*(t)-f^*(2t)]\frac{dt}{t}\\
&\le c \int_0^\infty
t^{1/(2p)}\left[\f_1^*\left(\frac{\sqrt{t}}{2}\right)
+\f_2^*\left(\frac{\sqrt{t}}{2}\right)\right]\frac{dt}{t}\\
&= c'\int_0^\infty
t^{1/p}\left[\f_1^*(t)+\f_2^*(t)\right]\frac{dt}{t}.
\end{aligned}
$$
This gives (\ref{bound}).

Set now
$$
f_h(x,y)=\frac1{h^2}\int_0^h\int_0^hf(x+u,y+v)\,du\,dv\quad (h>0)
$$
and
$$
g_h(x,y)=f(x,y)-f_h(x,y).
$$
We have (see (\ref{infty}))
$$
|f_h(x,y)|\le \min\left(\frac1h\int_0^h \psi_1(y+v)\,dv, \frac1h\int_0^h \psi_2(x+u)\,du\right).
$$
As above, this implies that $f_h\in S_0(\R^2)$ and thus $g_h\in S_0(\R^2)$ for any $h>0.$

We shall estimate $||g_h||_\infty.$ First,
$$
\begin{aligned}
|g_h(x,y)|&\le \frac1{h^2}\int_0^h\int_0^h|f(x+u,y+v)-f(x,y+v)|\,du\,dv\\
&+\frac1{h}\int_0^h |f(x,y+v)-f(x,y)|\,dv\\
&\le h^{1/p}\left[\frac1h\int_0^h\f_1(y+v)\,dv+\f_2(x)\right] \le
h^{1/p}\left[\f_1^{**}(h)+\f_2(x)\right].
\end{aligned}
$$
Similarly,
$$
|g_h(x,y)|\le h^{1/p}\left[\f_2^{**}(h)+\f_1(y)\right].
$$
There exists a set $E_h\subset \R^2$ of type $F_\s$ such that
$\mes_2 E_h\ge h^2$ and
$$
|g_h(x,y)|\ge g_h^*(h^2)\quad\mbox{for all}\quad (x,y)\in E_h.
$$
By the estimates obtained above,
\begin{equation}\label{new}
g_h^*(h^2)\le
h^{1/p}\left[\f_1^{**}(h)+\f_2^{**}(h)\right]+h^{1/p}\min\left(\f_1(y),\f_2(x)\right)
\end{equation}
for any $(x,y)\in E_h.$ At least one of the projections of $E_h$
onto the coordinate axes  has the one-dimensional measure not
smaller than $h$. If the projection $\Pi'(E_h)$ onto the $x-$axis
has this property, then
$$
\inf_{x\in \Pi'(E_h)} \f_2(x)\le \f_2^*(h).
$$
Similarly,
$$
\inf_{y\in \Pi''(E_h)} \f_1(y)\le \f_1^*(h)
$$
if $\mes_1 \Pi''(E_h)\ge h,$ where $\Pi''(E_h)$ is the projection
of $E_h$ onto the $y-$axis. Thus, using (\ref{new}), we obtain
\begin{equation}\label{3_3}
g_h^*(h^2)\le 2h^{1/p}[\f_1^{**}(h)+\f_2^{**}(h)].
\end{equation}
It follows from (\ref{rear}) that
\begin{equation}\label{3_4}
g_h^*(0+)-g_h^*(h^2)\le \frac{1}{\ln
2}\int_0^{h^2}[g_h^*(t)-g_h^*(2t)]\frac{dt}{t}.
\end{equation}

Further, we have the following estimates
\begin{equation}\label{3_100}
||f_h(\cdot,y)||^*_{\operatorname{Lip}\frac1p}\le \f_1^{**}(h)
\quad\mbox{and}\quad
||f_h(x,\cdot)||^*_{\operatorname{Lip}\frac1p}\le \f_2^{**}(h).
\end{equation}
Indeed, for any $\tau>0$
$$
\begin{aligned}
&|f_h(x+\tau,y)-f_h(x,y)|\\
&\le \frac1{h^2}\int_0^h\int_0^h|f(x+u+\tau,y+v)-f(x+u,y+v)|\,du\,dv\\
&\le \frac{\tau^{1/p}}{h}\int_0^h\f_1(y+v)\,dv\le
\tau^{1/p}\f_1^{**}(h).
\end{aligned}
$$
This implies the first inequality in (\ref{3_100}); the second
inequality is obtained similarly. Applying (\ref{3_100}), we get
$$
\begin{aligned}
||g_h(\cdot,y)||^*_{\operatorname{Lip}\frac1p}&\le
||f(\cdot,y)||^*_{\operatorname{Lip}\frac1p}+||f_h(\cdot,y)||^*_{\operatorname{Lip}\frac1p}\\
&\le \f_1(y)+\f_1^{**}(h),
\end{aligned}
$$
and similarly
$$
||g_h(x,\cdot)||^*_{\operatorname{Lip}\frac1p}\le
\f_2(x)+\f_2^{**}(h).
$$
Using these estimates and applying the same reasonings as in the
proof of
 (\ref{3_1}), we have
$$
g_h^*(t)-g_h^*(2t)\le
ct^{1/(2p)}\left[\f_1^*\left(\frac{\sqrt{t}}{2}\right)+\f_2^*\left(\frac{\sqrt{t}}{2}\right)+\f_1^{**}(h)+\f_2^{**}(h)\right].
$$
This inequality, (\ref{3_3}), and (\ref{3_4}) yield that
$$
\begin{aligned}
||g_h||_\infty &=g_h^*(0+)\le 2h^{1/p}[\f_1^{**}(h)+\f_2^{**}(h)]\\
&+c[\f_1^{**}(h)+\f_2^{**}(h)]\int_0^{h^2}t^{1/(2p)-1}dt\\
&+c\int_0^{h^2}t^{1/(2p)}\left[\f_1^*\left(\frac{\sqrt{t}}{2}\right)+\f_2^*\left(\frac{\sqrt{t}}{2}\right)\right]\frac{dt}{t}\\
&\le c'\left(h^{1/p}[\f_1^{**}(h)+\f_2^{**}(h)]+
\int_0^h t^{1/p}[\f_1^*(t)+\f_2^*(t)]\frac{dt}{t}\right)\\
&\le c''\int_0^ht^{1/p}[\f_1^*(t)+\f_2^*(t)]\frac{dt}{t}.
\end{aligned}
$$
It follows that $||g_h||_\infty\to 0$ as $h\to 0.$ Thus,
$f_h(x,y)$ converges uniformly on $\R^2$ as $h\to 0$, and the
limit function $\bar{f}$ is continuous on $\R^2.$ By the Lebesgue differentiation theorem, $f=\bar{f}$ almost everywhere. Further,
$$
\o(\bar{f};h)\le ||g_h||_\infty+\o(f_h;h).
$$
Applying (\ref{3_100}), we easily get that
$$
\o(f_h;h)\le h^{1/p}\left[\f_1^{**}(h)+\f_2^{**}(h)\right].
$$
Using this inequality and estimate of $||g_h||_\infty$ obtained
above, we have
$$
\o(\bar{f};h)\le
c\int_0^ht^{1/p}[\f_1^*(t)+\f_2^*(t)]\frac{dt}{t}.
$$
The proof is completed.
\end{proof}

In Theorem \ref{Second} the exterior $L^{p,1}-$norm cannot be
replaced by a weaker Lorentz norm. More exactly, the following
statement holds.
\begin{prop}\label{Optimality} Let $1< p<\infty$ and $1<q<\infty.$ Then there exists a function
\begin{equation}\label{optim1}
f\in
L^{p,q}\left[\operatorname{Lip}\frac1p\right]_{\operatorname{sym}}
\end{equation}
such that $f\not\in L^\infty(\R^2).$
\end{prop}
\begin{proof} Choose $0<\beta<1-1/q$. Let
$$
g(x,y)=\left|\ln\frac{4}{|x|+|y|}\right|^\beta\quad\mbox{if}\quad
(x,y)\not=(0,0),\,\,g(0,0)=0.
$$
Further, let $\f\in C_0^\infty(\R)$, $\f(t)=1$ if $|t|\le 1/2,$
and $\f(t)=0$ if $|t|\ge 1$. Set
$$
f(x,y)=g(x,y)\f(|x|+|y|).
$$
Then $f\not\in L^\infty(\R^2).$ We shall prove that
\begin{equation}\label{optim2}
f\in L^{p,q}_x\left[\left(\operatorname{Lip}
\frac1p\right)_y\right].
\end{equation}
Denote for $h\in (0,1]$
$$
\psi(x,h)=\sup_{0<y\le 1}h^{-1/p}|g(x,y+h)-g(x,y)|\chi_{(0,1]}(x).
$$
It is easy to see that (\ref{optim2}) will be proved if we show
that the function
$$
\psi(x)=\sup_{0<h\le
1}\psi(x,h)\chi_{(0,1]}(x)
$$
 belongs to $L^{p,q}(\R).$

Fix $x\in (0,1]$. Since the function $g(x,y)$ is concave with
respect to $y$ on the interval $[0,2],$ we have that
\begin{equation}\label{optim3}
\psi(x,h)=h^{-1/p}\left[\left(\ln\frac4x\right)^\beta-\left(\ln\frac{4}{x+h}\right)^\beta\right],
\,\, h\in(0,1].
\end{equation}
If $0<h\le x,$ then
\begin{equation}\label{optim4}
\psi(x,h)\le
\frac{4\b}{x}h^{1-1/p}\left(\ln\frac2x\right)^{\beta-1}\le 4\b
x^{-1/p}\left(\ln\frac2x\right)^{\beta-1}.
\end{equation}
Let now $x\le h\le 1.$ Set $u=h/x;$ then $1\le u\le 1/x.$ By
(\ref{optim3}), we have
$$
\begin{aligned}
&\psi(x,h)\le (xu)^{-1/p}\left[\left(\ln\frac4x\right)^\beta-\left(\ln\frac{4}{x(1+u)}\right)^\beta\right]\\
&=(xu)^{-1/p}(z^\b-(z-\ln(1+u))^\b)= x^{-1/p}z^\b
u^{-1/p}\left[1-\left(1-\frac{\ln(1+u)}{z}\right)^\b\right],
\end{aligned}
$$
where
$$
z=\ln\frac4x, \quad z>\ln(1+u).
$$
Since
 $$
 1-(1-\tau)^\b\le c \tau     \quad (0<\b<1,\,0< \tau<1),
 $$
 where $c$ is a  constant depending only on $\b,$ we obtain
 $$
u^{-1/p}\left[1-\left(1-\frac{\ln(1+u)}{z}\right)^\b\right]\le
\frac{cu^{-1/p}\ln(1+u)}{z}\le \frac{c'}{z}.
$$
Using this estimate and taking into account (\ref{optim4}), we
have that  for all $0<h\le 1$
$$
\psi(x)=\sup_{0<h\le 1}\psi(x,h)\chi_{(0,1]}(x)\le
c x^{-1/p}\left(\ln\frac2x\right)^{\beta-1}.
$$
It follows that $\psi\in L^{p,q}(\R)$, and we obtain
(\ref{optim2}). Since $f(x,y)=f(y,x),$ this implies
(\ref{optim1}).

\end{proof}

\section{Embeddings}

For a function $\f\in L^\infty(\R),$ denote
$$
\o(f;t)=\sup_{0\le h\le t}||\Delta_h\f||_\infty.
$$
If, in addition, $\f\in S_0(\R),$ then for any $t>0$
\begin{equation}\label{L_infty}
||\f||_\infty\le \f^*(t)+2\o(\f;t).
\end{equation}
Indeed, for any $\e>0$ the set
$$
E_\e=\{x\in
\R:|\f(x)|>||\f||_\infty-\e\}
$$
has a positive measure. By the definition of the non-increasing
\newline rearrangement, for any $x\in E_\e$ there exists $h\in (0,2t)$
such that $$|\f(x+h)|\le \f^*(t).$$ Thus,
$$
|\f(x)|\le |\f(x)-\f(x+h)| + \f^*(t)\le \o(\f;2t)+ \f^*(t).
$$
This implies (\ref{L_infty}).

\begin{teo}\label{First} Let $1\le p<q<\infty.$ Then
$\ \mathcal U_p(\R^2)\subset \mathcal U_q(\R^2).$ Moreover, for
any $f\in \mathcal U_p(\R^2)$
\begin{equation}\label{first}
||f||_{\mathcal U_q(\R^2)}\le c||f||_{\mathcal U_p(\R^2)}.
\end{equation}
\end{teo}
\begin{proof} For any $r\ge 1,$ denote
$$
\f_{r,1}(y)= \sup_{h>0}h^{-1/r}||\Delta_h f_y||_\infty
\quad\mbox{and}\quad \f_{r,2}(x)= \sup_{h>0}h^{-1/r}||\Delta_h
f_x||_\infty.
$$
We estimate $\f_{q,1}^*(t).$ First, let $0<h\le t.$ Then
$$
|f(x+h,y)-f(x,y)|\le\f_{p,1}(y)h^{1/p} \le
\f_{p,1}(y)h^{1/q}t^{1/p-1/q}.
$$
Thus,
\begin{equation}\label{1}
\sup_{0<h\le t}h^{-1/q}||\Delta_h f_y||_\infty\le
\f_{p,1}(y)t^{1/p-1/q}.
\end{equation}
In particular, we have that
$$
\sup_{0<h\le 1} h^{-1/q}||\Delta_h f_y||_\infty\le \f_{p,1}(y).
$$
On the other hand,
$$
\sup_{h\ge 1}h^{-1/q}||\Delta_h f_y||_\infty\le 2||f_y||_\infty.
$$
Thus,
$$
\f_{q,1}(y)\le 2||f_y||_{\operatorname{Lip}\frac1p}.
$$
Since the function on the right-hand side belongs to
$L^{p,1}(\R),$ we have that $\f_{q,1}\in S_0(\R).$

Let now $h>t.$  For any fixed $y\in\R,$ we have, applying
(\ref{L_infty})
$$
||\Delta_h f_y||_\infty\le (\Delta_hf_y)^*(t)+2\o(f_y;t)
$$
\begin{equation}\label{h>t}
\le
(\Delta_hf_y)^*(t)+2\f_{p,1}(y)t^{1/p}.
\end{equation}
We shall estimate the first term on the right-hand side. For any
$\tau>0$
$$
\begin{aligned}
\Delta_h f_y(x) &=|f(x+h,y)-f(x,y)|\le |f(x+h,y)-f(x+h,y+\tau)|\\
&+ |f(x,y)-f(x,y+\tau)|+|f(x+h,y+\tau)-f(x,y+\tau)|\\
&\le \left[\f_{p,2}(x+h) + \f_{p,2}(x)\right]\tau^{1/p}+
\f_{q,1}(y+\tau)h^{1/q}.
\end{aligned}
$$
Thus,
$$
(\Delta_hf_y)^*(t)\le 2\f_{p,2}^*(t/2)\tau^{1/p}+
\f_{q,1}(y+\tau)h^{1/q}\quad(\tau>0).
$$
For any fixed $y$ there exists $\tau\in(0,4t]$ such that
$\f_{q,1}(y+\tau)\le \f_{q,1}^*(2t).$ Taking this $\tau,$ we
obtain
$$
(\Delta_hf_y)^*(t)\le
8\f_{p,2}^*(t/2)t^{1/p}+\f_{q,1}^*(2t)h^{1/q}.
$$
From here and (\ref{h>t}),
$$
\sup_{h\ge t}h^{-1/q}||\Delta_h f_y||_\infty \le
2\left[\f_{p,1}(y)+4\f_{p,2}^*(t/2)\right]t^{1/p-1/q}+\f_{q,1}^*(2t).
$$
This inequality and (\ref{1}) imply that
$$
\f_{q,1}(y) \le 4\left[\f_{p,1}(y)+2\f_{p,2}^*(t/2)\right]t^{1/p-1/q}+\f_{q,1}^*(2t)
$$
and therefore
$$
\f_{q,1}^*(t)-\f_{q,1}^*(2t) \le
4\left[\f_{p,1}^*(t)+2\f_{p,2}^*(t/2)\right]t^{1/p-1/q}.
$$
Thus,
$$
\int_0^\infty\left[\f_{q,1}^*(t)-\f_{q,1}^*(2t)\right]
t^{1/q-1}\,dt
$$
\begin{equation}\label{2}
\le
4\int_0^\infty\f_{p,1}^*(t)t^{1/p-1}\,dt+16\int_0^\infty\f_{p,2}^*(t)t^{1/p-1}\,dt.
\end{equation}
Since $\f_{q,1}\in S_0(\R),$ we have by (\ref{rear})
$$
\begin{aligned}
\int_0^\infty \f_{q,1}^*(t)t^{1/q-1}\,dt&\le 2\int_0^\infty
t^{1/q-1}
\int_{t/2}^\infty\left[\f_{q,1}^*(u)-\f_{q,1}^*(2u)\right]\frac{du}{u}\,dt\\
&\le
2^{1+1/q}q\int_0^\infty\left[\f_{q,1}^*(t)-\f_{q,1}^*(2t)\right]
t^{1/q-1}\,dt.
\end{aligned}
$$
Together with (\ref{2}), this implies that
$$
||\f_{q,1}||_{L^{q,1}}\le 2^6
q\left(||\f_{p,1}||_{L^{p,1}}+||\f_{p,2}||_{L^{p,1}}\right).
$$
Clearly, a similar estimate holds for $||\f_{q,2}||_{L^{q,1}}.$
Thus, we have
\begin{equation}\label{lipschitz}
||\f_{q,1}||_{L^{q,1}}+||\f_{q,2}||_{L^{q,1}}\le
2^7q\left(||\f_{p,1}||_{L^{p,1}}+||\f_{p,2}||_{L^{p,1}}\right).
\end{equation}

Further, let
$$
 \psi_1(y)= ||f_y||_\infty, \quad \psi_2(x)=
||f_x||_{\infty}.
$$
Then
$$
\int_1^\infty t^{1/q-1}[\psi^*_1(t)+\psi_2^*(t)]\,dt\le
\int_1^\infty t^{1/p-1}[\psi^*_1(t)+\psi_2^*(t)]\,dt
$$
and, by (\ref{bound}),
$$
\int_0^1  t^{1/q-1}[\psi^*_1(t)+\psi_2^*(t)]\,dt\le
2q||f||_\infty\le cq \left(||\f_1||_{p,1}+||\f_2||_{p,1}\right).
$$
These estimates together with (\ref{lipschitz}) imply
(\ref{first}).
\end{proof}

Finally, we  prove embedding (\ref{embed_1}).
\begin{prop}\label{EMBED} For any function $f\in W_1^{2,2}(\R^2)$
\begin{equation} \label{final1}
\int_\R||f_x||_{\operatorname{Lip} 1}^*\,dx\le
\frac12||D_{2}^2f||_1, \,\, \int_\R||f_y||_{\operatorname{Lip}
1}^*\,dy\le \frac12||D_{1}^2f||_1,
\end{equation}
and
\begin{equation}\label{final2}
||f||_{\mathcal U_1(\R^2)}\le
c||f||_1^{1/2}(||D_{1}^2f||_1^{1/2}+||D_{2}^2f||_1^{1/2}).
\end{equation}
\end{prop}
\begin{proof} Let $f\in W_1^{2,2}(\R^2)$. Then by Gagliardo-Nirenberg inequalities (see \cite{Ga}, \cite{Nir}), the
first order weak derivatives $D_1f$ and $D_2f$ exist and
\begin{equation}\label{multip}
||D_1 f||_1\le c||f||_1^{1/2}||D_{1}^2f||_1^{1/2}, \quad ||D_2
f||_1\le c||f||_1^{1/2}||D_{2}^2f||_1^{1/2}.
\end{equation}

For almost all $x\in \R$ we have
$$
||f_x||_{\operatorname{Lip} 1}^*\le ||D_2 f(x,\cdot)||_\infty\le
\frac12\int_\R |D_{2}^2 f(x,y)|\,dy.
$$
This implies the first inequality in (\ref{final1}); the second
inequality follows in the same way.

Further, for almost all $x\in \R$
$$
||f_x||_\infty\le \int_\R|D_2 f(x,y)|\, dy.
$$
Thus, by (\ref{multip}),
$$
\int_\R||f_x||_\infty\,dx\le c ||D_2 f||_1\le
c||f||_1^{1/2}||D_{2}^2f||_1^{1/2}.
$$
Similarly,
$$
\int_\R||f_y||_\infty\,dy\le c||f||_1^{1/2}||D_{1}^2f||_1^{1/2}.
$$
These estimates together with (\ref{final1}) imply (\ref{final2}).
\end{proof}

\end{document}